\documentclass[oneside,english]{amsart}
\usepackage[T1]{fontenc}
\usepackage[latin9]{inputenc}
\usepackage{color}
\usepackage{units}
\usepackage{textcomp}
\usepackage{amsthm}
\usepackage{amssymb}
\usepackage{esint}

\makeatletter
\numberwithin{equation}{section}
\numberwithin{figure}{section}
\theoremstyle{plain}
\newtheorem{thm}{\protect\theoremname}
  \theoremstyle{plain}
  \newtheorem{prop}[thm]{\protect\propositionname}
  \theoremstyle{plain}
  \newtheorem{lem}[thm]{\protect\lemmaname}
  \theoremstyle{plain}
  \newtheorem{cor}[thm]{\protect\corollaryname}
  \theoremstyle{plain}
  \newtheorem*{prop*}{\protect\propositionname}
  \theoremstyle{remark}
  \newtheorem{rem}[thm]{\protect\remarkname}

\makeatother

\usepackage{babel}
  \providecommand{\corollaryname}{Corollary}
  \providecommand{\lemmaname}{Lemma}
  \providecommand{\propositionname}{Proposition}
  \providecommand{\remarkname}{Remark}
\providecommand{\theoremname}{Theorem}

\global\long\def\SL#1{SL\left(#1,\mathbb{R}\right)}
\global\long\def\norm#1{\left\Vert #1\right\Vert }

\global\long\def\g{\mathfrak{g}}

\global\long\def\mr{\mathbb{R}}

\global\long\def\ma{\mathfrak{a}}

\global\long\def\gng{\nicefrac{G}{\Gamma}}

\global\long\def\hc{\mathcal{H}}

\global\long\def\lkgk{L^{1}\left(K\backslash G/K\right)}
\global\long\def\Vol{\mbox{vol}}
\global\long\def\dm{\mbox{diag}}

\begin{document}
\title[Lattice points in higher rank groups]{Counting Lattice Points in norm balls \\ on higher rank 
 simple Lie groups}
\author{ Alexander Gorodnik, Amos Nevo and Gal Yehoshua} 
\address{School of Mathematics, University of Bristol, Bristol UK }
\email{a.gorodnik@bristol.ac.uk}
\address{Department of Mathematics, Technion IIT, Israel}
\email{anevo@tx.technion.ac.il}
\address{Department of Mathematics, Technion IIT, Israel}
\email{gal.yehoshua@gmail.com}

\date{\today}
\thanks{The first author acknowledges support of ERC grant 239606. The second author acknowledges
  support of ISF grant 2095/15.}

\begin{abstract}
We establish an error estimate for counting lattice points in Euclidean norm balls (associated to an arbitrary irreducible linear representation) for lattices in simple Lie groups of real rank at least two. Our approach utilizes  refined spectral estimates based on the existence of universal pointwise bounds for spherical functions on the groups involved. We focus particularly on the case of the special linear groups where we give a detailed proof of error estimates which constitute the first improvement of the best current bound established by Duke, Rudnick and Sarnak in 1991, and are nearly twice as good in some cases. 
\end{abstract}

\maketitle

{\small
\tableofcontents
}

\section{The lattice point counting problem in higher rank simple groups\label{sub:Statement_of_results}}

Let $G$ denote a connected non-compact simple Lie group with finite center, and $\Gamma$  a lattice in $G$, namely discrete subgroup of finite covolume. Let $\Vol$ denote the Haar measure on $G$ normalized so that $\Gamma$ has co-volume $1$. 
We denote by $K$ a maximal compact subgroup of $G$. Let $\tau\colon G\rightarrow GL\left(N,\mr\right)$
be a non-trivial  irreducible representation of $G$ on $N$-dimensional Euclidean
space. We assume (without lost of generality) that $\tau\left(K\right)\subset SO\left(N\right)$.
Let $\norm{\cdot}^2$ denote the Euclidean norm $tr\left(A^{t}A\right)$
on $GL\left(N,\mr\right)$, and let $\norm g_{\tau}^{2}=\norm{\tau\left(g\right)}^{2}=\mbox{tr}\left(\tau\left(g\right)^{t}\tau\left(g\right)\right)$.
We will consider balls of radius $T$ in $G$ with respect to $\norm{\cdot}_{\tau}$,
namely: 
\begin{equation}
B_{T}^{\tau}=\left\{ g\in G\colon\norm g_{\tau}\leq T\right\} .\label{eq:domains}
\end{equation}
Note that $B_{T}^{\tau}$ is invariant under left and right translations
by $K$, and we will call sets satisfying this condition bi-$K$-invariant, or radial 
sets. In the present paper we will study the lattice point counting
problem in the balls $B_{T}^{\tau}$, namely we will aim to establish
an asymptotic formula for $\left|B_{T}^{\tau}\cap\Gamma\right|$ in
the form: 
\[
\frac{\left|B_{T}^{\tau}\cap\Gamma\right|}{\Vol\left(B_{T}^{\tau}\right)}=1+O\left(\Vol(B_{T}^{\tau})^{-\kappa}\right)\,,
\]
with $\kappa>0$ as large as possible, and $T\ge T_0> 0$. Before we state our main results,
let us recall what is currently the best exponent known for the error term in the higher rank case, established
by Duke, Rudnick and Sarnak. For the group $G=\SL{n+1}$ with $n\geq2$, and the balls $B_{T}^{\tau}$, it is as follows.
\begin{thm}
\cite[Thm. 3.1]{DRS} \label{DRS}Let $\Gamma\subset\SL{n+1}$ be
any lattice. Then for $T$ sufficiently large, and any $\eta>0$,   

\begin{eqnarray}
\frac{\left|B_{T}^{\tau}\cap\Gamma\right|}{\Vol\left(B_{T}^{\tau}\right)} & = & 1+O_\eta\left(\Vol(B_{T}^{\tau})^{-\frac{1}{n(n+1)(n+2)}+\eta}\right).\label{eq:DRS-EQ}
\end{eqnarray}

\end{thm}
We will denote this exponent by $\kappa_{0}=\kappa_0(n)=\frac{1}{n(n+1)(n+2)}$. 

Our purpose in the present paper is to improve this exponent for a large collection
of families $B_{T}^{\tau}$, associated with suitable irreducible
representations of $\SL{n+1}$. We note that our method applies in
principle to any connected higher-rank simple Lie group with finite
center, but for simplicity of exposition we will concentrate below only on the case
of $G=\SL{n+1}$ with $n\geq2$. Let us begin by stating a special
case of our main result and comparing it to the exponent cited above.
\begin{thm}\label{adjoint}
\label{thm:adjoin_estimate}Let $G=\SL{n+1}$, $n\geq2$ and $\Gamma\subset G$
be any lattice. Consider the adjoint representation of $\SL{n+1}$. Then, for $T\ge T_0$

\begin{enumerate}
\item For $n$ odd: 
\begin{eqnarray*}
\frac{\left|B_{T}^{\mbox{Ad}}\cap\Gamma\right|}{\Vol\left(B_{T}^{\mbox{Ad}}\right)} & = & 1+O\left(\Vol\left(B_{T}^{\mbox{Ad}}\right)^{-2\frac{n}{n+1}\kappa_{0}}\left(\log T\right)^{q}\right).
\end{eqnarray*}

\item For $n$ even: 
\begin{eqnarray*}
\frac{\left|B_{T}^{\mbox{Ad}}\cap\Gamma\right|}{\Vol\left(B_{T}^{\mbox{Ad}}\right)} & = & 1+O\left(\Vol\left(B_{T}^{\mbox{Ad}}\right)^{-2\frac{n}{n+2}\kappa_{0}}\left(\log T\right)^{q}\right).
\end{eqnarray*}

\end{enumerate}

In both cases $q$ and $T_0$ are positive numbers which depends on $n$ but not on the lattice, and can be made explicit. 
\end{thm}
Thus, for large $n$, the exponent established above is nearly twice as large as the exponent $\kappa_0$ established by \cite{DRS} for the error term. The first case where the estimate is improved is for $G=SL_4(\mathbb{R})$.  

Anticipating our results below, let us note that we will establish an error term exponent  for Euclidean norm balls associated with any irreducible representation $\tau$ of $G$. These representations are classified by dominant weights, and for a certain non-empty cone of dominant weights we will establish the exponent stated in Theorem  \ref{adjoint}. For dominant weights belonging to certain other cones, we will establish an exponent which improves on $\kappa_0$ but is smaller than the one stated in Theorem \ref{adjoint}. 
\section{Preliminaries and notation}

\subsection{Lie algebras : roots and weights of $\SL{n+1}$}\label{sub:Lie-algebras}

In the present section we establish notation and record some preliminaries. Our discussion is based on \cite{Humphreys}. 

Let $G=\SL{n+1}$, $K=SO\left(n+1,\mr\right)$ a maximal compact subgroup
and let $A$ be the following $\mathbb{R}$-split Cartan subgroup, 
$$A=\left\{ \mbox{diag}\left(a_{1},...,a_{n+1}\right)\colon\prod a_{i}=1\mbox{ and}\hspace{1em}a_{i}>0,\hspace{1em}\text{for    } 1\leq i\leq n+1\right\}\,.$$
Let $\g$ be the Lie algebra of $G$ and $\ma$ of $A$. Let $\left(X,Y\right)_{\g}=\mbox{tr}\left(\mbox{ad}_{\g}\left(X\right)\mbox{ad}_{\g}\left(Y\right)\right)$
be the Killing form on $\mathfrak{g}$. For $\gamma,\delta\in\ma^{*}=\text{Hom}(\ma,\mathbb{R})$
we set $\left\langle \gamma,\delta\right\rangle =\frac{2\left(\gamma,\delta\right)}{\left(\delta,\delta\right)}$,
where $\left(\cdot,\cdot\right)$ is the form induced from $\left(\cdot,\cdot\right)_{\g}$.
Let $\Phi\subset\ma^{*}$ be the root system for the pair $\left(\mathfrak{g,a}\right)$.
Fix a choice of simple roots $\Delta=\left\{ \alpha_{1},...,\alpha_{n}\right\} \subset\Phi$
and denote the set of positive roots by $\Phi^{+}$. Denote by $\ma^{+}=\left\{ H\in\ma\colon\alpha_{i}\left(H\right)\geq0,\,\,1\leq i\leq n\right\} $
the non-negative Weyl chamber with respect to $\Delta$. Let $\left\{ \tilde{\beta}_{j}\right\} _{j=1}^{n}\subset\ma^{+}$
be defined by $\alpha_{i}\left(\tilde{\beta}_{j}\right)=\delta_{i,j}$
$1\leq i,j\leq n$, namely the dual basis of the simple roots. An
element $\gamma\in\ma^{*}$ is called a weight if the numbers $\left\langle \gamma,\alpha\right\rangle $
are integers for all $\alpha\in\Phi$, and denote by $\Lambda$ the set
of all weights. For $\gamma\in\Lambda$, if the integers $\left\langle \gamma,\alpha\right\rangle $
are non-negative for all $\alpha\in\Delta$, then the weight is called
dominant. Let $\Lambda^{+}$ denote the set of dominant weights. We
denote by $\lambda_{i}\,,$ $1\leq i\leq n$ the fundamental weights,
namely those satisfying the equations $\left\langle \lambda_{i},\alpha_{j}\right\rangle =\delta_{i,j}$
$1\leq i,j\leq n$. Another example of a dominant weight is half the
sum of the positive roots, denoted $\rho$, which is equal to ${\displaystyle \sum_{i=1}^{n}\lambda_{i}}$.
We recall that finite dimensional irreducible representations of $\SL{n+1}$,
are in a bijective correspondence with dominant weights, see e.g.  \cite[Theorem 5.5]{Knapp-Green}.

\subsection{Volumes of radial balls  in $\SL{n+1}$\label{sub:Volumes}}

Every $g\in G=\SL{n+1}$ can be written as $g=k_{1}ak_{2}$ where
$k_{i}\in K$ and $a\in A^{+}$. This decomposition yield the 
integration formula \cite[p.142]{Knapp} :
\begin{prop}
Given a Haar measure on $G$, 
for $f\in C_{c}\left(\SL{n+1}\right)$ we have, 
\begin{eqnarray}
 &  & \int_{\SL{n+1}}f\left(g\right)dg=\label{eq:IntegralFormula}\\
 &  & \int_{K\times\mathfrak{a}^{+}\times K}f\left(k_{1}\exp\left(H\right)k_{2}\right)\prod_{\alpha\in\Phi^{+}}\sinh\left(\alpha\left(H\right)\right)dk_{1}dHdk_{2}\nonumber 
\end{eqnarray}
where $dH$ is a suitable scalar multiple of the Lebesgue measure on $\ma$,
and $dk$ is the Haar probability measure on the maximal compact subgroup
$K$.
\end{prop}
This formula was used in \cite{GW,M} to compute the asymptotic volume
of general norm balls in connected noncompact semisimple Lie groups. The computation applies in particular to the balls we investigate. Let $\tau$, $B_{T}^{\tau}$,
$\Gamma$ and $\Vol$ be as in §\ref{sub:Statement_of_results}. Let
$\lambda$ be the highest weight for the representation $\tau$. Applying
\cite[Thm 2.7]{GW} (or equivalently \cite[Corollary 1.1]{M}) to
$B_{T}^{\tau}\subset G$ we get the following asymptotics, 
\[
\Vol\left(B_{T}^{\tau}\right)\sim C_{1}\left(\log T\right)^{l}T^{\frac{1}{m_{1}}}\,,
\]
where $l\in\mathbb{N}$ and $C_{1}>0$ is a constant independent of
$T$. The rate of growth, namely $m_{1}$, is given by the following expression : 
\begin{equation}
{\displaystyle m_{1}=\min_{j\in\left\{ 1,...,n\right\} }\frac{\lambda\left(\tilde{\beta}_{j}\right)}{2\rho\left(\tilde{\beta}_{j}\right)}}.\label{eq:m1}
\end{equation}
Let $I=I(\lambda)=\left\{ 1\leq i\leq n\colon m_{1}=\frac{\lambda\left(\tilde{\beta}_{i}\right)}{2\rho\left(\tilde{\beta}_{i}\right)}\right\} $
be the set of minimizing indices, namely the set of indices where the minimum is obtained.

\section{Averaging operators and counting lattice points}

\subsection{The spectral method of counting lattice points}\label{sec:spectral}


Let $G=\SL{n+1}$, $\Gamma\subset G$ any lattice and $B_{T}^{\tau}$
be as in §\ref{sub:Statement_of_results}. Let $\pi_{\nicefrac{G}{\Gamma}}$
be the unitary representation of $G$ on $L^{2}\left(\nicefrac{G}{\Gamma}\right)$, given by 
$\left(\pi_{\nicefrac{G}{\Gamma}}\left(g\right)\tilde{f}\right)\left(h\Gamma\right)=\tilde{f}\left(g^{-1}h\Gamma\right),\hspace{1em}\forall\tilde{f}\in L^{2}\left(\nicefrac{G}{\Gamma}\right),h\in\nicefrac{G}{\Gamma}$.
Let $b_{T}^{\tau}$ denote the normalized indicator function, ${\displaystyle \frac{\chi_{B_{T}^{\tau}}}{\Vol\left(B_{T}^{\tau}\right)}}$.
The averaging operators $\pi_{\nicefrac{G}{\Gamma}}\left(b_{T}^{\tau}\right)$
are defined by,
\[
\left(\pi_{\nicefrac{G}{\Gamma}}\left(b_{T}^{\tau}\right)\left(\tilde{f}\right)\right)\left(x\right)=\frac{1}{\Vol\left(B_{T}^{\tau}\right)}\int_{B_{T}^{\tau}}
\pi_{\nicefrac{G}{\Gamma}}\left(g\right)\tilde{f}\left(x\Gamma\right)dg,\hspace{1em}\forall\tilde{f}\in L^{2}\left(\nicefrac{G}{\Gamma}\right).
\]
We let $L_{0}^{2}\left(\nicefrac{G}{\Gamma}\right)$ denote the space
of $L^{2}$-functions on $\nicefrac{G}{\Gamma}$ with zero integral,
and we let $\pi_{\nicefrac{G}{\Gamma}}^{0}$ denote the restriction
of the representation $\pi_{\nicefrac{G}{\Gamma}}$ to $L^{2}_0\left(\nicefrac{G}{\Gamma}\right)$.

We will use \cite[Theorem 1.9]{GN} to establish our estimate of the
error term. To apply this result to our families of balls
$B_{T}^{\tau}$ it is enough to show that the following two conditions are satisfied. 
\begin{enumerate}
\item The families are Lipschitz admissible in the sense of \cite[ Theorem 3.15]{GN10}. 
\item The averaging operators $\pi_{\nicefrac{G}{\Gamma}}\left(b_{T}^{\tau}\right)$
satisfy the quantitative mean ergodic theorem, with rate function given as a negative power of the volume.
\end{enumerate}
Condition 1 is explained and established in \cite[ Theorem 3.15]{GN10}. The arguments
for showing Condition 2 are of spectral nature, and we turn to explain how to exploit the spherical spectrum of $L_{0}^{2}\left(\nicefrac{G}{\Gamma}\right)$ for this purpose. 

\subsection{Spectral estimates of spherical functions}
We will use concepts from the theory of Gelfand pairs and the theory
of Banach $\ast$-algebras, and for a general exposition of this theory we refer to
\cite{Wolf,Folland}. 

Let $G$ be a connected simple Lie group with
a finite center,  $K\subset G$ a maximal compact subgroup. It is well
known that $\left(G,K\right)$ is a Gelfand pair, so $L^{1}\left(K\backslash G/K\right)$
is a commutative Banach $\ast$-algebra. The map $f\mapsto f^{*}$ defined
by, $f^{*}\left(x\right)=\overline{f\left(x^{-1}\right)}$ is the
involution of $L^{1}\left(K\backslash G/K\right)$. Denote by $\Sigma$
the Gelfand spectrum of $\lkgk$. We can identify $\Sigma$ with the set of bounded
$\left(G,K\right)$-spherical functions $\omega$ on $G$ \cite[Theorem 8.2.7]{Wolf}.
Using this identification the Gelfand transform is given by $\hat{f}\left(\omega\right)=\int_{G}f\left(g\right)\omega\left(g^{-1}\right)dg$.
We denote by $\Sigma^{+}\subset\Sigma$ the subset of positive definite
spherical functions for $\left(G,K\right)$. Consider any unitary
representation $\pi\colon G\rightarrow\mathcal{U}\left(\mathcal{H}\right)$
with no invariant unit vectors. Such a representation defines a nondegenrate
$\ast$-representation of $\lkgk$ on $\hc$. This representation is defined
by $f\mapsto\pi\left(f\right)=\int_{G}f\left(g\right)\pi\left(g\right)dg$.
By the spectral theorem of $\ast$-representations \cite[Theorem 1.54]{Folland}
there is a unique regular projection-valued measure, denoted $P^{\pi}$, on
the spectrum of the algebra, such that 
$P^{\pi}$ is supported on $\Sigma^{+}$, and the following  formula holds : 
\begin{equation}
\forall f\in\lkgk\hspace{1em}\left\langle \pi\left(f\right)u,v\right\rangle =\int_{\Sigma^{+}}\hat{f}\left(\omega\right)dP_{u,v}^{\pi}\left(\omega\right)\,,\label{eq:spectal_thm}
\end{equation}
where $P^\pi_{u,v}$ is the scalar complex bounded Borel measure on the spectrum $\Sigma^+$ determined by the pair of vectors $u,v$  and the projection valued measure $P^\pi$. Therefore 
\begin{eqnarray}
\norm{\pi\left(f\right)}^{2} & = & \sup_{\norm v=1}\left\langle \pi\left(f\right)v,\pi\left(f\right)v\right\rangle \nonumber \\
 & = & \sup_{\norm v=1}\left\langle \pi\left(f^{*}*f\right)v,v\right\rangle \nonumber \\
 & = & \sup_{\norm v=1}\int_{\Sigma^{+}}\widehat{f^{*}*f}\left(\omega\right)dP_{v,v}^{\pi}\left(\omega\right)\label{eq:norm_est_pec_measure}\\
 & = & \sup_{\norm v=1}\int_{\Sigma^{+}}\widehat{f^{*}}\left(\omega\right)\cdot\hat{f}\left(\omega\right)dP_{v,v}^{\pi}\left(\omega\right)\nonumber \\
 & = & \sup_{\norm v=1}\int_{\Sigma^{+}}\left|\hat{f}\left(\omega\right)\right|^{2}dP_{v,v}^{\pi}\left(\omega\right).\nonumber 
\end{eqnarray}

A positive definite spherical function for $\left(G,K\right)$
arise as the matrix coefficient associated with the unique $K$-invariant unit vector  of a uniquely determined irreducible unitary representation \cite[Theorem 8.4.8]{Wolf}. Our approach is based on a remarkable uniform spectral estimate, which is a special feature of the spherical unitary representation theory of simple Lie groups of real rank at least $2$. Namely, we will use the fact that all the non-constant positive definite spherical functions 
can be {\it bounded by one and the same positive function} on the group. Any such bounding function $F$, which is refered to  as a universal pointwise bound, gives a norm bound on all bi-$K$-invariant averaging operators on the group, as follows.

\begin{thm}
Let $F$ be an upper bound for all matrix coefficients associated with $K$-invariant unit vectors 
of irreducible non-trivial unitary representations. Let $\pi$ be any unitary representation
without invariant unit vectors, as above. Then, for any bi-$K$-invariant function $f\in L^1(G)$ 
\begin{equation}
\norm{\pi\left(f\right)}\leq\int_{G}\left|f\left(g\right)\right|F\left(g\right)dg.\label{eq:norm_estimate}
\end{equation}
\end{thm}
\begin{proof}
Let $\omega$ be a non-constant positive definite spherical function. Let $\pi_\omega$ be the unique irreducible unitary representation of $G$ and
$v_{\omega}$ the unique (up to scalar)  $K$-fixed cyclic unit vector (see \cite[Theorem 8.4.8]{Wolf})
satisfying $\omega(g)=\left\langle v_\omega,\pi_\omega(g)v_\omega\right\rangle$. 

Since $F$ is a universal pointwise bound, for every $\omega\in\Sigma^{+}$
we have %
\[
\left|\omega\left(g^{-1}\right)\right|=\left|\left\langle v_{\omega},\pi_{\omega}\left(g^{-1}\right)v_{\omega}\right\rangle \right|=\left|\left\langle \pi_{\omega}\left(g\right)v_{\omega},v_{\omega}\right\rangle \right|\leq F\left(g\right).
\]
Thus $\left|\hat{f}\left(\omega\right)\right|=\left|\int_{G} f\left(g\right)\omega\left(g^{-1}\right)dg\right|\leq\int_{G}\left|f\left(g\right)\right|F\left(g\right)dg$
for every $\omega\in\Sigma^{+}$. Substituting this into equation \eqref{eq:norm_est_pec_measure}
gives,
\begin{eqnarray*}
\norm{\pi\left(f\right)}^{2} & = & \sup_{\norm v=1}\int_{\Sigma^{+}}\left|\hat{f}\left(\omega\right)\right|^{2}dP_{v,v}^{\pi}\left(\omega\right)\\
 & \leq & \sup_{\norm v=1}\int_{\Sigma^{+}}\left(\int_{G}\left|f\left(g\right)\right|F\left(g\right)dg\right)^{2}dP_{v,v}^{\pi}\left(\omega\right)\\
 & = & \left(\int_{G}\left|f\left(g\right)\right|F\left(g\right)dg\right)^{2}.
\end{eqnarray*}

\end{proof}

\subsection{The integrability bound}
The error estimate in the lattice point counting problem provided by Duke, Rudnick and Sarnak can be derived using the following argument. 
Suppose that each non-trivial positive definite spherical function $\omega$ on $G$ 
belong to $L^{p+\eta}(G) $ for
any $\eta>0$, and that the $L^{p+\eta}(G)$-norm of $\omega$ is uniformly bounded. This is certainly the case if there exists a universal pointwise bound $F$ which is in $L^{p+\eta}(G)$ for every $\eta > 0$. 
Then, applying \eqref{eq:norm_estimate} to $\pi_{\nicefrac{G}{\Gamma}}^{0}$
and $b_{T}^{\tau}$, we get the following inequality:
\begin{eqnarray*}
\norm{\pi_{\nicefrac{G}{\Gamma}}^{0}\left(b_{T}^{\tau}\right)\left(\tilde{f}\right)}_{L^{2}\left(\gng\right)} & \leq & \norm{\pi_{\nicefrac{G}{\Gamma}}^{0}\left(b_{T}^{\tau}\right)}\norm{\tilde{f}}_{L^{2}\left(\gng\right)}\\
 & \leq & \left(\frac{1}{\Vol\left(B_{T}^{\tau}\right)}\int_{B_{T}^{\tau}}F\left(g\right)dg\right)\norm{\tilde{f}}_{L^{2}\left(\gng\right)}.
\end{eqnarray*}
This inequality gives us a quantitative mean ergodic theorem. Hence
the conditions mentioned in §\ref{sec:spectral} are met,
and we can apply \cite[Theorem 1.9]{GN} to write, for $T\ge T_0 >0$ :
\begin{equation}
\left|\frac{\left|B_{T}^{\tau}\cap\Gamma\right|}{\Vol\left(B_{T}^{\tau}\right)}-1\right|<C \left(\frac{1}{\Vol\left(B_{T}^{\tau}\right)}\int_{B_{T}^{\tau}}F\left(g\right)dg\right)^{\frac{1}{1+d}} \,.\label{eq:Estimate}
\end{equation}
with $d=\dim\left(\nicefrac{G}{K}\right)=\frac{\left(n+1\right)\left(n+2\right)}{2}-1$ (see e.g.  \cite[remark 1.10]{GN}).

Thus the quality of the error estimate depends on the upper bound for the integral. One possibility is to use the $L^{p+\eta}(G)$ integrability condition for $F$, so that using H\"older's inequality :
$$\le C_\eta \left(\frac{1}{\Vol\left(B_{T}^{\tau}\right)}\right)^{\frac{1}{p(d+1)}-\eta}\norm{F}_{L^{p+\eta}(G)}^{1/(d+1)}\,,
$$
where $C_\eta$ is a computable positive constant. This estimate, namely $\kappa_0=\frac{1}{p(d+1)}$ is the one established in \cite{DRS}, using the fact that for $G=SL(n+1,\mathbb{R})$ the exponent of integrability is $p=2n$. 

In the next section we will focus on $SL({n+1},\mathbb{R})$ and show how to derive a better upper estimate for the integral 
$\int_{B_{T}^{\tau}}F\left(g\right)dg$, for certain functions $F$, thus improving the error estimate arising from the exponent of integrability of $F$.

\section{Universal pointwise bounds for $\SL{n+1}$\label{sub:Pointwise-uniform-bounds}}

\subsection{Universal pointwise bounds}
Let $G=\SL{n+1}$, $K=SO\left(n+1,\mr\right)$ and $F\in C_{c}\left(G\right)$
 a bi-K-invariant function. By the integration formula \eqref{eq:IntegralFormula}
we have,
\begin{eqnarray*}
\int_{G}F\left(g\right)dg & = & \int_{K\times\mathfrak{a}^{+}\times K}F\left(k_{1}\exp\left(H\right)k_{2}\right)\prod_{\alpha\in\Phi^{+}}\sinh\left(\alpha\left(H\right)\right)dk_{1}dHdk_{2}\\
 & = & \int_{\mathfrak{a}^{+}}F\left(\exp\left(H\right)\right)\prod_{\alpha\in\Phi^{+}}\sinh\left(\alpha\left(H\right)\right)dH.
\end{eqnarray*}
Hence we can consider $F$ to be a function on $\ma^{+}$.
The functions we will use as universal bounds for the positive definite spherical
functions for $\left(G,K\right)$, will all have the following 
form. For all $H\in\ma^{+}$ 
\begin{equation}\label{universal-bound}
F_\theta \left(H\right)=P\left(H\right)e^{-\theta\left(H\right)}\text{ for some }\theta\in\ma^{*}\,\, \text{ where }\theta (H)> 0 \,\,\,\forall H\in\ma^{+}.
\end{equation}
Here $P$ is a positive function which can be bounded by a polynomial function in $\norm{H}$. 

Three distinct functions which can serve as bounds for the
positive definite spherical functions of $\left(G,K\right)$ are, first, a suitable root of the Harish Chandra $\Xi_G$-function, (see e.g. \cite{CHH}),
second, a function constructed by Howe and Tan (see \cite[theorem 3.3.12]{HoweTan}), and third, a sharper version of it constructed by Oh (see \cite{Oh}). Let us turn to describe them in greater detail. 

\subsubsection{Harish-Chandra's $\Xi_G$-function} It is a well-known fact that for a suitable $n=n_G$ the function $\Xi_{G}^{\frac{1}{n}}$ is a bound
for the non-constant positive definite spherical functions of $\left(G,K\right)$, see the discussion in 
\cite{CHH}. Furthermore the Harish-Chandra function satisfies (see \cite[Theorem 4.6.4]{GV})
\[
\Xi_{G}\left(e^{H}\right)\leq C\left(1+\norm H\right)^{\left|\Phi^{+}\right|}e^{-\rho\left(H\right)},\hspace{1em}\forall H\in\ma^{+}.
\]
 For $H\in\ma^{+}$, write $H=\dm\left(h_{1},...,h_{n+1}\right)$,
where $h_{i}\geq h_{i+1}$ and $\sum_{i=1}^{n+1}h_{i}=0$. Then the first universal pointwise bound is given by 
\[
F_{\rho/n}(H)= C\left(1+\norm H\right)^{\left|\Phi^{+}\right|}e^{-\frac{\rho\left(H\right)}{n_G}}
\]
so that the linear function $\theta$ in this case is $\rho/n_G$. The constant $n_G$ has been computed explicitly for all simple Lie groups of real rank at least two, and can be taken at the least integer $k$ such that all non-constant positive-definite spherical functions on $G$ are in $L^{2k+\eta}(G)$ for every $\eta > 0$. For $G=SL(n+1,\mathbb{R})$ it is equal to $n$. 

\subsubsection{Howe-Tan's function.}
Howe and Tan \cite[theorem 3.3.12]{HoweTan} showed that the bi-K-invariant function : 
\begin{eqnarray*}
k_{1}\exp\left(H\right)k_{2}\,\,\, \mapsto\,\,\, \min_{i\neq j}\,\,\Xi_{\SL 2}\left(\exp\begin{pmatrix}\frac{h_{i}-h_{j}}{2} & 0\\
0 & \frac{h_{j}-h_{i}}{2}
\end{pmatrix}\right),
\end{eqnarray*}
is a bound for all the non-constant positive definite spherical functions of $\left(G,K\right)$.
We recall that ${\displaystyle \Xi_{\SL 2}\left(\begin{pmatrix}a & 0\\
0 & a^{-1}
\end{pmatrix}\right)\sim\frac{\log a}{a}},\hspace{1em}a>0$. Using this asymptotic of $\Xi_{\SL 2}$, and the fact that $h_{i}\geq h_{i+1}$
we conclude that a second universal pointwise bound is provided by the function $F_{\beta/2}$ given by  
\[
F_{\beta/2}(H)= C\cdot\beta\left(H\right)\cdot e^{-\frac{1}{2}\beta\left(H\right)},
\]
where $\beta=\sum_{i=1}^{n}\alpha_{i}$ is the highest root, $\beta\left(H\right)=h_{1}-h_{n+1}$.

\subsubsection{Oh's function}
Using the same spectral approach more efficiently, by utilizing strongly orthogonal systems, Oh \cite{Oh} showed that in (\ref{universal-bound}) we can take the linear functional $\theta=\gamma$ given explicitly as follows. 

\[
\gamma=\begin{cases}
\frac{1}{2}\left(\sum_{i=1}^{\nicefrac{\left(n-1\right)}{2}}i\alpha_{i}+\sum_{i=\nicefrac{\left(n+1\right)}{2}}^{n}\left(n+1-i\right)\alpha_{i}\right) & n\mbox{ odd}\mbox{ }\\
\frac{1}{2}\left(\sum_{i=1}^{\nicefrac{n}{2}}i\alpha_{i}+\frac{n}{2}\alpha_{\nicefrac{n}{2}+1}+\sum_{i=\nicefrac{n}{2}+2}^{n}\left(n+1-i\right)\alpha_{i}\right) & n\mbox{ even}.
\end{cases}
\]
Thus a third universal pointwise bound is $F_\gamma(H)=P(H)e^{-\gamma(H)}$, 
where $P(H)$ is bounded by an explicit polynomial in $\norm{H}$. 

\begin{rem}
\begin{enumerate}
\item We note that for $n=2$ namely for $G=SL(3,\mathbb{R})$, the functions constructed by Howe-Tan and by Oh are the same. As we shall see below, it follows that this case is the only one for which we will not achieve an improvement of the error estimate in the lattice point counting problem. 
\item As will become apparent in the next section, all three functions discussed above on $G=SL(n+1,\mathbb{R})$ satisfy that they are in $L^p(G)$ if and only if $p > 2n$. 
\end{enumerate}
\end{rem}
\subsection{Estimating integrals of universal pointwise bounds}
Our task now is to estimate the integral of the universal pointwise bound on norm balls $B_T^\tau$. This amounts to bounding the integral of a polynomial times 
an exponential function on suitable regions of Euclidean space.  

Therefore consider, in Euclidean space $\mathbb{R}^k$,  the region :
$$D=\left\{ \left(t_{1},...,t_{k}\right)\colon\forall i,\hspace{1em}t_{i}\geq0,\hspace{1em}\sum_{i=1}^{k}m_{i}t_{i}\leq S\right\} \,.$$
 We begin with the following   
\begin{lem}
\label{lem:expo_int}Given $k\in\mathbb{N}$ and $0<m_{1}\leq...\leq m_{k}$,
for all $S>0$,
\[
\int_{D}P\left(t_{1},...,t_{k}\right)e^{\sum_{i=1}^{k}t_{i}}dt_{1}...dt_{k}\leq Ce^{\frac{S}{m_{1}}}S^{\deg\left(P\right)+k-1},
\]
for suitable $C$, where $P$ is any polynomial function.  \end{lem}
\begin{proof}
Since $P$ is a polynomial function,  there exists $c_{1}>0$ such that
for all $S>0$, $P\mid_{D}\leq c_{1}S^{\deg\left(P\right)}$. Next, consider the following re-parametrization of $D$. Set $0\le t_k\le S/m_k$ and for $1\le j \le k-1$: 
\[
0\leq t_{j}\leq\frac{1}{m_{j}}\left(S-\sum_{i=j+1}^{k}m_{i}t_{i}\right).
\]
This allows us to write,
\begin{eqnarray*}
 &  & \int_{D}P\left(\left(t_{1},...,t_{k}\right)\right)e^{\sum_{i=1}^{k}t_{i}}dt_{1}...dt_{k}\\
 & = & \int_{t_{k}=0}^{\frac{S}{m_{k}}}...\int_{t_{1}=0}^{\frac{1}{m_{1}}\left(S-\sum_{i=2}^{k}m_{i}t_{i}\right)}P\left(t_{1},...,t_{k}\right)e^{\sum_{i=1}^{k}t_{i}}dt_{1}...dt_{k}\\
 & \leq & c_{1}S^{\deg P}\int_{t_{k}=0}^{\frac{S}{m_{k}}}...\int_{t_{1}=0}^{\frac{1}{m_{1}}\left(S-\sum_{i=2}^{k}m_{i}t_{i}\right)}e^{\sum_{i=1}^{k}t_{i}}dt_{1}...dt_{k}\\
 & \leq & c_{1}S^{\deg P}\int_{t_{k}=0}^{\frac{S}{m_{k}}}...\int_{t_{2}=0}^{\frac{1}{m_{2}}\left(S-\sum_{i=3}^{k}m_{i}t_{i}\right)}\left(e^{\frac{1}{m_{1}}\left(S-\sum_{i=2}^{k}m_{i}t_{i}\right)}-1\right)e^{\sum_{i=2}^{k}t_{i}}dt_{2}...dt_{k}\\
 & \leq & c_{1}S^{\deg P}e^{\frac{S}{m_{1}}}\int_{t_{k}=0}^{\frac{S}{m_{k}}}...\int_{t_{2}=0}^{\frac{1}{m_{2}}\left(S-\sum_{i=3}^{k}m_{i}t_{i}\right)}e^{\sum_{i=2}^{k}t_{i}\left(1-\frac{m_{i}}{m_{1}}\right)}dt_{2}...dt_{k}\\
 & \leq & c_{1}S^{\deg P}e^{\frac{S}{m_{1}}}\int_{t_{k}=0}^{\frac{S}{m_{k}}}...\int_{t_{2}=0}^{\frac{1}{m_{2}}\left(S-\sum_{i=3}^{k}m_{i}t_{i}\right)}1dt_{2}...dt_{k}\leq Ce^{\frac{S}{m_{1}}}S^{\deg P+k-1}.
\end{eqnarray*}
\end{proof}
Let us now apply this fact to the integrals we are interested in bounding. As in \S\ref{sub:Statement_of_results}, we let $\tau$ denote an irreducible representation associated with a dominant weight $\lambda$, which is its highest weight. 
\begin{cor}
Let $G=\SL{n+1}$, and $B_{T}^{\tau}$ be as in \S\ref{sub:Statement_of_results}.
 If $F$ is a bi-K-invariant
function and $F\mid_{\ma^{+}}$ is of the form $F_\theta\left(H\right)=P\left(H\right)e^{-\theta \left(H\right)}$, then 
\[
\int_{B_{T}^{\tau}}F\left(g\right)dg\leq C_{1}^{\prime}\left(\log T\right)^{l^{\prime}}T^{\frac{1}{m_{1}^{\prime}}}\,,
\]
where $l^{\prime}\in\mathbb{N},C_{1}^{\prime}>0$. Furthermore, setting $\psi=2\rho-\theta$, the following formula for $m_{1}^{\prime}$ holds :

\begin{equation}
{\displaystyle m_{1}^{\prime}=\min_{1\leq i\leq n}\left(\frac{\lambda\left(\tilde{\beta_{i}}\right)}{\psi\left(\tilde{\beta_{i}}\right)}\right)}.\label{eq:m1Prime}
\end{equation}
\end{cor}
Extending the definition of $I(\lambda)$ in \S 2.2 above, let us denote by $I^{\prime}(\lambda)$ the set
of minimizing indices in the equation above. This set depends on the functional $\theta$ chosen to define the universal estimate, but we will suppress this dependence in the notation.  
\begin{proof}
By equation \eqref{eq:IntegralFormula} we have,
\[
\int_{B_{T}^{\tau}}F\left(g\right)dg=\int_{\mathfrak{a}^{+}\left(T,\tau\right)}F\mid_{\ma^{+}}\left(H\right)\prod_{\alpha\in\Phi^{+}}\sinh\left(\alpha\left(H\right)\right)dH,
\]
where $\mathfrak{a}^{+}\left(T,\tau\right)=\left\{ H\in\mathfrak{a}^{+}\colon\norm{\tau\left(\exp\left(H\right)\right)}\leq T\right\} $.
Let $\lambda$ be the highest weight of $\tau$. Comparing with the
max-norm, we find that there is $\tilde{C}$ such that for all $T>0$,
\begin{eqnarray*}
\mathfrak{a}^{+}\left(T,\tau\right) & \subset & \left\{ H\in\mathfrak{a}^{+}\colon e^{\lambda\left(H\right)}\leq\frac{T}{\tilde{C}}\right\} \\
 & = & \left\{ H\in\mathfrak{a}^{+}\colon\lambda\left(H\right)\leq\log T-\log\tilde{C}\right\} .
\end{eqnarray*}
We denote $D_{T,\tau}=\left\{ H\in\mathfrak{a}^{+}\colon\lambda\left(H\right)\leq\log T-\log\tilde{C}\right\} $.
Since $\sinh\left(x\right)\leq e^{x}$, we have the following inequality,
\begin{eqnarray*}
\int_{B_{T}^{\tau}}F_\theta\left(g\right)dg & = & \int_{\mathfrak{a}^{+}\left(T,\tau\right)}P\left(H\right)e^{-\theta\left(H\right)}\prod_{\alpha\in\Phi^{+}}\sinh\left(\alpha\left(H\right)\right)dH\\
 & \leq & C_1\int_{D_{T,\tau}}P\left(H\right)e^{\left(2\rho-\theta\right)\left(H\right)}dH
\end{eqnarray*}
Let ${\displaystyle H=\sum_{j=1}^{n}t_{j}\frac{\tilde{\beta_{j}}}{\psi\left(\tilde{\beta_{j}}\right)}}$.
In terms of these coordinates we have (recall $\psi=2\rho-\theta$): 
\[
D_{T,\tau}=\left\{ \left(t_{1},...,t_{n}\right)\colon\forall i,\hspace{1em}t_{i}\geq0,\hspace{1em}\sum_{j=1}^{n}\frac{\lambda\left(\tilde{\beta_{j}}\right)}{\psi\left(\tilde{\beta_{j}}\right)}t_{j}\leq\log T-\log\tilde{C}\right\} ,
\]
and 
\[
\int_{B_{T}^{\tau}}F_\theta\left(g\right)dg\leq C_2\int_{D_{T,\tau}}P\left(t_{1},...,t_{n}\right)e^{\sum_{j=1}^{n}t_{j}}\prod_{j=1}^{n}dt_{j}.
\]
Thus using Lemma \ref{lem:expo_int} we find that,
\[
\int_{B_{T}^{\tau}}F_\theta\left(g\right)dg\leq C_1^\prime\left(\log T\right)^{l^{\prime}}T^{\frac{1}{m_{1}^{\prime}}}.
\]
With ${\displaystyle m_{1}^{\prime}=\min_{1\leq j\leq n}\frac{\lambda\left(\tilde{\beta_{j}}\right)}{\psi\left(\tilde{\beta_{j}}\right)}}$,
$C_1^\prime>0$ and $l^{\prime}\in\mathbb{N}$.
\end{proof}
Recall the parameter $m_1$ defined in equation (\ref{eq:m1}). Given $m_1$ and $m_1^\prime$ (which depends also on the choice of the functional $\theta$ defining the universal pointwise bound),  and the parameter $\kappa_0=\kappa_0(n)$, and using the foregoing formula we can rewrite equation \eqref{eq:Estimate} and state our main error estimate as follows. 
\begin{cor} 
\begin{equation}
\frac{\left|B_{T}^{\tau}\cap\Gamma\right|}{\Vol\left(B_{T}^{\tau}\right)}=1+O\left(\Vol\left(B_{T}^{\tau}\right)^{-2n\left(1-\frac{m_{1}}{m_{1}^{\prime}}\right)\kappa_{0}}\left(\log T\right)^{q}\right),\label{eq:RealEstimate}
\end{equation}
for some $q\in\mathbb{N}$ and $T \ge T_0$ that can be computed explicitly. 
\end{cor}
This asymptotic formula gives a solution to the lattice point counting problem in $\left|B_{T}^{\tau}\cap\Gamma\right|$, with exponent $\kappa=\kappa(\tau)=2n\left(1-\frac{m_{1}}{m_{1}^{\prime}}\right)\kappa_{0}$. 
The exponent $\kappa(\tau)$ is determined by the underlying representation $\tau$ and the functional $\theta$ defining the universal pointwise estimate $F_\theta$.
Our next task therefore is to estimate the exponent just described as $\tau$ ranges over the set of irreducible finite-dimensional representations, namely over the space of dominant weights, and establish when does $\kappa(\tau)$  constitute an improvement over $\kappa_0$.  

\section{Improving the error estimates}
\subsection{Comparing exponents}
In order to estimate the exponent from equation \eqref{eq:RealEstimate}
for various irreducible representations, let us note the following. The largest exponent is achieved when using the functional $\gamma$ described above,
and we will presently give conditions on the irreducible representations for which 
this largest exponent can be established.

First, given a dominant weight $\lambda$ recall the set $I=I(\lambda)$ of minimizing indices defined in \S \ref{sub:Volumes}. 
\begin{prop}
\label{prop:m1/m1'_lower_bound} Let $G=\SL{n+1}$, $\lambda\in\Lambda^{+}$
and $\theta\in\ma^{*}$. For every $i\in I=I(\lambda)$ we have ${\displaystyle \frac{m_{1}}{m_{1}^{\prime}}\geq\frac{\psi\left(\tilde{\beta}_{i}\right)}{2\rho\left(\tilde{\beta}_{i}\right)}}$,
so in particular ${\displaystyle {\displaystyle \frac{m_{1}}{m_{1}^{\prime}}}\geq\min_{1\leq j\leq n}\frac{\psi\left(\tilde{\beta}_{j}\right)}{2\rho\left(\tilde{\beta}_{j}\right)}}$.\end{prop}
\begin{proof}
By definition, for any $1\le i\le n$ 
$${\displaystyle m_{1}^{\prime}=\min_{1\leq j\leq n}\left(\frac{\lambda\left(\tilde{\beta_{j}}\right)}{\psi\left(\tilde{\beta_{j}}\right)}\right)\leq\frac{\lambda\left(\tilde{\beta_{i}}\right)}{\psi\left(\tilde{\beta_{i}}\right)}}\,.$$
Thus for every $i\in I=I(\lambda)$ 
\[
\frac{m_{1}}{m_{1}^{\prime}}=\frac{\lambda\left(\tilde{\beta}_{i}\right)}{2\rho\left(\tilde{\beta}_{i}\right)}\left(m_{1}^{\prime}\right)^{-1}\geq\frac{\lambda\left(\tilde{\beta}_{i}\right)}{2\rho\left(\tilde{\beta}_{i}\right)}\frac{\psi\left(\tilde{\beta}_{i}\right)}{\lambda\left(\tilde{\beta}_{i}\right)}=\frac{\psi\left(\tilde{\beta}_{i}\right)}{2\rho\left(\tilde{\beta}_{i}\right)}.
\]

\end{proof}

We can now easily deduce that using the functional $\rho/n_G$ as a universal pointwise bound does not improve the error estimate.   
\begin{prop}
Let $\lambda\in\Lambda^{+}$, and let $\tau$ be the irreducible representation
of $\SL{n+1}$ associated with $\lambda$. Then the 
exponent associated with $\tau$  and the linear functional $\frac{\rho}{n_G}$ is equal to $\kappa_{0}$.
\end{prop}
\begin{proof}
Indeed in this case $\psi=2\rho-\frac{\rho}{n}=2\rho\left(1-\frac{1}{2n}\right)$.
Thus for any $\lambda\in\Lambda^{+}$, 
\[
m_{1}^{\prime}={\displaystyle \min_{1\leq j\leq n}\left(\frac{\lambda\left(\tilde{\beta_{j}}\right)}{\psi\left(\tilde{\beta_{j}}\right)}\right)=\frac{1}{1-\frac{1}{2n}}\min_{1\leq j\leq n}\left(\frac{\lambda\left(\tilde{\beta_{j}}\right)}{2\rho\left(\tilde{\beta_{j}}\right)}\right)=\frac{1}{1-\frac{1}{2n}}m_{1}}.
\]
This means that $\kappa=2n\left(1-\frac{m_{1}}{m_{1}^{\prime}}\right)\kappa_{0}=2n\left(1-\left(1-\frac{1}{2n}\right)\right)\kappa_{0}=\kappa_{0}$. \end{proof}

The same phenomenon arises with the universal bound defined by the linear functional $\beta/2$, $\beta$ the highest root. 
\begin{prop}
Let $\lambda\in\Lambda^{+}$, and let $\tau$ be the irreducible representation
of $\SL{n+1}$ associated with $\lambda$. Then the
exponent associated with $\tau$ and the linear functional $\frac{1}{2}\beta$ is less then
or equals to $\kappa_{0}$.\end{prop}
\begin{proof}
Recall that $2\rho=\sum_{k=1}^{n}\alpha_{k}k\left(n+1-k\right)$,
thus $2\rho\left(\tilde{\beta}_{j}\right)=j\left(n+1-j\right)$. We
also have $\beta=\sum_{k=1}^{n}\alpha_{k}$, so that $\beta\left(\tilde{\beta}_{j}\right)=1$.
Now using proposition \ref{prop:m1/m1'_lower_bound}: 
\begin{eqnarray*}
\frac{m_{1}}{m_{1}^{\prime}} & \geq & \min_{1\leq j\leq n}\frac{\psi\left(\tilde{\beta}_{j}\right)}{2\rho\left(\tilde{\beta}_{j}\right)}=\min_{1\leq j\leq n}\left(\frac{2\rho\left(\tilde{\beta}_{j}\right)-\frac{1}{2}\beta\left(\tilde{\beta}_{j}\right)}{2\rho\left(\tilde{\beta}_{j}\right)}\right)\\
 & = & \min_{1\leq j\leq n}\left(1-\frac{1}{4\rho\left(\tilde{\beta}_{j}\right)}\right)=1-\left(\max_{1\leq j\leq n}\frac{1}{4\rho\left(\tilde{\beta}_{j}\right)}\right)\\
 & = & 1-\frac{1}{4\min_{1\leq j\leq n}\rho\left(\tilde{\beta}_{j}\right)}=1-\frac{1}{2n}.
\end{eqnarray*}
This means that $\kappa=2n\left(1-\frac{m_{1}}{m_{1}^{\prime}}\right)\kappa_{0}\leq2n\left(1-\left(1-\frac{1}{2n}\right)\right)\kappa_{0}=\kappa_{0}$.
\end{proof}
Next we consider the universal pointwise bound defined by the linear functional $\gamma$ described above, and show that as we vary  $\lambda\in\Lambda^{+}$ a better error estimate can be established in many cases.

\subsection{Main result : improving the error estimate}

As usual, let $\tau$ be an irreducible representation
of $G=\SL{n+1}$ for $n\geq2$, and let $\lambda$ be the highest
weight. Let $B_{T}^{\tau}=\left\{ g\in G:\norm{\tau\left(g\right)}\leq T\right\} $. $\gamma$ be the functional defined in \S \ref{sub:Pointwise-uniform-bounds} 
and let $F_\gamma$ be the universal pointwise bound defined there. 
Let $m_{1},m_{1}^{\prime}$ be defined by equations \eqref{eq:m1} and \eqref{eq:m1Prime}. Let 
$I=I(\lambda)$ and $I^{\prime}=I^\prime(\lambda)$ be the sets of minimizing indices defining them, with  $I^\prime(\lambda)$ defined by the choice $\theta=\gamma$. These choices give rise to the main result of the paper. 
\begin{thm}
\label{thm:Best_estimate}
Let $G=\SL{n+1}$ for $n\geq2$, let $\Gamma$ any lattice subgroup, and  $T\ge T_0$.  
\begin{enumerate}
\item If n is odd and $\frac{n+1}{2}\in I(\lambda)\cap I^{\prime}(\lambda)$, then 
\[
\frac{\left|B_{T}^{\tau}\cap\Gamma\right|}{\Vol\left(B_{T}^{\tau}\right)}=1+O\left(\Vol\left(B_{T}^{\tau}\right)^{-2\frac{n}{n+1}\kappa_{0}}\left(\log T\right)^{q}\right),
\]
\item If n is even and $\frac{n}{2}+1\in I(\lambda)\cap I^{\prime}(\lambda)\mbox{ or }\frac{n}{2}\in I(\lambda)\cap I^{\prime}(\lambda)$, 
then
\[
\frac{\left|B_{T}^{\tau}\cap\Gamma\right|}{\Vol\left(B_{T}^{\tau}\right)}=1+O\left(\Vol\left(B_{T}^{\tau}\right)^{-2\frac{n}{n+2}\kappa_{0}}\left(\log T\right)^{q}\right),
\]
\end{enumerate}
In both cases $q$ and $T_0$ are positive numbers which depends on $n$ but not on the lattice, and can be made explicit. 
\end{thm}
Thus when $n$ is odd the exponent associated with a dominant weight $\lambda$  is $\kappa=2\frac{n}{n+1}\kappa_{0}$, and 
when $n$ is even the exponent is $\kappa=2\frac{n}{n+2}\kappa_{0}$, provided that $\lambda$ satisfies the conditions stated above. In both cases, these values are the largest exponent that our method provides, and the exponents are nearly twice as large as $\kappa_0$. 

\begin{proof}

The largest exponent is achieved when the ratio ${\displaystyle \frac{m_{1}}{m_{1}^{\prime}}}$
is minimized. For the universal pointwise bound associated with $\gamma$, setting $\psi=2\rho-\gamma$ 
 by Proposition \ref{prop:m1/m1'_lower_bound} we have  ${\displaystyle {\displaystyle \frac{m_{1}}{m_{1}^{\prime}}}\geq\min_{1\leq i\leq n}\frac{\psi\left(\tilde{\beta}_{i}\right)}{2\rho\left(\tilde{\beta}_{i}\right)}}$.
To find conditions for which ${\displaystyle {\displaystyle \frac{m_{1}}{m_{1}^{\prime}}}=\min_{1\leq i\leq n}\frac{\psi\left(\tilde{\beta}_{i}\right)}{2\rho\left(\tilde{\beta}_{i}\right)}}$, we will use the following.  
\begin{prop}\label{minimum}
${\displaystyle \min_{1\leq i\leq n}\frac{\psi\left(\tilde{\beta}_{i}\right)}{2\rho\left(\tilde{\beta}_{i}\right)}=\begin{cases}
\frac{n}{n+1} & n\mbox{ is odd}\\
\frac{n+1}{n+2} & n\mbox{ is even}
\end{cases}}$. Moreover for n odd the minimum is attained at $\frac{n+1}{2}$,
and for n even at $\left\{ \frac{n}{2},\frac{n}{2}+1\right\} $. 
\end{prop}

\begin{proof} 
We demonstrate this for n odd, and a similar argument applies to the case when $n$ is even.
First note that $\psi\left(\tilde{\beta_{j}}\right)\mbox{ and }\rho\left(\tilde{\beta_{j}}\right)$
have the symmetry, $j\mapsto n+1-j$.  Therefore
\begin{eqnarray*}
\min_{1\leq i\leq n}\frac{\psi\left(\tilde{\beta}_{i}\right)}{2\rho\left(\tilde{\beta}_{i}\right)} & = & \min_{1\leq i\leq\frac{n+1}{2}}\frac{\psi\left(\tilde{\beta}_{i}\right)}{2\rho\left(\tilde{\beta}_{i}\right)}\\
 & = & 1+\min_{1\leq i\leq\frac{n+1}{2}}\frac{-\frac{1}{2}i}{i\left(n+1-i\right)}\\
 & = & 1-\frac{1}{2}\max_{1\leq i\leq\frac{n+1}{2}}\frac{1}{n+1-i}\\
 & = & 1-\frac{1}{n+1}=\frac{n}{n+1},
\end{eqnarray*}
and it is can be easily seen that the minimum is attained at ${\displaystyle \frac{n+1}{2}}$. 
\end{proof}

Resuming the proof of Theorem \ref{thm:Best_estimate}, we conclude that in the event that $n$ is odd, if $s=\frac{n+1}{2}\in I\cap I^{\prime}$
then, 
\[
\frac{m_{1}}{m_{1}^{\prime}}=\frac{\lambda\left(\tilde{\beta}_{s}\right)}{2\rho\left(\tilde{\beta}_{s}\right)}\frac{\psi\left(\tilde{\beta}_{s}\right)}{\lambda\left(\tilde{\beta}_{s}\right)}=\frac{\psi\left(\tilde{\beta}_{s}\right)}{2\rho\left(\tilde{\beta}_{s}\right)}=\min_{1\leq i\leq n}\frac{\psi\left(\tilde{\beta}_{i}\right)}{2\rho\left(\tilde{\beta}_{i}\right)}=\frac{n}{n+1}.
\]
This implies that the associated exponent will be $\kappa=2n\left(1-\frac{m_{1}}{m_{1}^{\prime}}\right)\kappa_{0}=\frac{2n}{n+1}\kappa_{0}$. 
This exponent is greater then $\kappa_{0}$ if $n$ is odd and $n \ge 3$. Similarly we handle the
case that $n$ is even.
\end{proof}

\subsection{Examples of admissible dominant weights}
Let us now show that there exist dominant weights for which the improved error estimate is achieved. Namely we give some examples for dominant weights that meet the conditions
of theorem \ref{thm:Best_estimate}.
\begin{thm}
Let $n$ be odd. Then every dominant weight $\mu$ belonging to the set 
$$W_n=\left\{ \lambda_{i}+\lambda_{n+1-i}:i\in\left\{ 1,...,\lfloor\frac{n+1}{4}\rfloor\right\} \right\}$$
satisfies the condition of Theorem  \ref{thm:Best_estimate}.
\end{thm}
\begin{proof}
We need to show that for weights $\mu\in W_n$, $\frac{n+1}{2}\in I\cap I^{\prime}$.
Setting $s=\frac{n+1}{2}$ we have to show,\textcolor{black}{{} }
\[
\frac{\mu\left(\tilde{\beta}_{s}\right)}{2\rho\left(\tilde{\beta}_{s}\right)}\leq\frac{\mu\left(\tilde{\beta}_{j}\right)}{2\rho\left(\tilde{\beta}_{j}\right)}\mbox{ and }\frac{\mu\left(\tilde{\beta}_{s}\right)}{\psi\left(\tilde{\beta}_{s}\right)}\leq\frac{\mu\left(\tilde{\beta}_{j}\right)}{\psi\left(\tilde{\beta}_{j}\right)}\,,\,\forall j\neq s.
\]
\textcolor{black}{as noted above}
\[
\frac{\psi\left(\tilde{\beta}_{s}\right)}{2\rho\left(\tilde{\beta}_{s}\right)}\leq\frac{\psi\left(\tilde{\beta}_{j}\right)}{2\rho\left(\tilde{\beta}_{j}\right)}\quad\forall j\neq s\Longrightarrow\frac{\psi\left(\tilde{\beta}_{s}\right)}{\psi\left(\tilde{\beta}_{j}\right)}\leq\frac{2\rho\left(\tilde{\beta}_{s}\right)}{2\rho\left(\tilde{\beta}_{j}\right)}\quad\forall j\neq s.
\]
Hence to show that $\mu\in W_n $ satisfies the property stated in Theorem  \ref{thm:Best_estimate} it is enough to show,
\begin{equation}
\frac{\mu\left(\tilde{\beta}_{s}\right)}{\psi\left(\tilde{\beta}_{s}\right)}\leq\frac{\mu\left(\tilde{\beta}_{j}\right)}{\psi\left(\tilde{\beta}_{j}\right)}\quad\forall j\neq s
\end{equation}
Recall that the fundamental weights satisfy $\lambda_{i}\left(\tilde{\beta_{j}}\right)=\left(C_{n}^{-1}\right)_{i,j}$
where $C_{n}$ is the Cartan matrix of the root system. Hence if $\mu=\sum_{k=1}^{n}q_{k}\lambda_{k}$,
then
\[
C_{n}^{-1}\begin{pmatrix}q_{1}\\
\vdots\\
q_{n}
\end{pmatrix}=\begin{pmatrix}\mu\left(\tilde{\beta_{1}}\right)\\
\vdots\\
\mu\left(\tilde{\beta_{n}}\right)
\end{pmatrix}.
\]
Using the above equation it can be calculated that 
\[
\left(\lambda_{1}+\lambda_{n}\right)\left(\tilde{\beta_{j}}\right)=\left(1,1,...,1\right),
\]
\[
\left(\lambda_{2}+\lambda_{n-1}\right)\left(\tilde{\beta_{j}}\right)=\left(1,2,2,...,2,1\right),
\]
and in general
\[
\left(\lambda_{i}+\lambda_{n+1-i}\right)\left(\tilde{\beta_{j}}\right)=\left(1,2,3,...,\underbrace{i}_{\mbox{i'th entry}},i,...,\underbrace{i}_{\mbox{n+1-i'th entry}},...,2,1\right).
\]
Next let us recall that by definition of $\gamma$ 
\[
\psi=2\rho-\gamma=\sum_{i=1}^{\frac{n-1}{2}}i\left(n+\frac{1}{2}-i\right)\alpha_{i}+\sum_{i=\frac{n+1}{2}}^{n}\left(n+1-i\right)\left(i-\frac{1}{2}\right)\alpha_{i}.
\]
Recall that $\psi$ has the symmetry given by $\psi\left(\tilde{\beta_{j}}\right)=\psi\left(\tilde{\beta}_{n+1-j}\right)$
and note that the same is true for every weight in $W_n$. Hence it is 
enough to show that $\frac{\mu\left(\tilde{\beta}_{s}\right)}{\psi\left(\tilde{\beta}_{s}\right)}\leq\frac{\mu\left(\tilde{\beta}_{j}\right)}{\psi\left(\tilde{\beta}_{j}\right)}\quad\forall j<s.$
Let $\mu=\lambda_{i}+\lambda_{n+1-i}$. Then
\[
\frac{\mu\left(\tilde{\beta}_{j}\right)}{\psi\left(\tilde{\beta}_{j}\right)}=\begin{cases}
\frac{j}{j\left(n+\frac{1}{2}-j\right)} & j<i\\
\frac{i}{j\left(n+\frac{1}{2}-j\right)} & j\geq i
\end{cases}=\begin{cases}
\frac{1}{\left(n+\frac{1}{2}-j\right)} & j<i\\
\frac{i}{j\left(n+\frac{1}{2}-j\right)} & j\geq i
\end{cases}.
\]
Thus
\begin{eqnarray*}
\min_{1\leq j\leq s}\frac{\mu\left(\tilde{\beta}_{j}\right)}{\psi\left(\tilde{\beta}_{j}\right)} & = & \min\left(\min_{1\leq j<i}\left(\frac{1}{\left(n+\frac{1}{2}-j\right)}\right),\min_{i\leq j\leq s}\left(\frac{i}{j\left(n+\frac{1}{2}-j\right)}\right)\right)\\
 & = & \min\left(\frac{1}{n-\frac{1}{2}},\frac{i}{s\left(s-\frac{1}{2}\right)}\right).
\end{eqnarray*}
We would like the minimum to be achieved at $j=s$, and this occurs when 
\[
\frac{i}{s\left(s-\frac{1}{2}\right)}\leq\frac{1}{n-\frac{1}{2}}\iff i\leq\frac{n}{n-\frac{1}{2}}\frac{n+1}{4}.
\]
This obviously holds for $i<\frac{n+1}{4}$, and so we have shown
that for any $\mu\mbox{ in }W_n$ the condition stated in Theorem  \ref{thm:Best_estimate} 
(for $n$ odd) holds. Therefore the balls defined by the representations associated with these dominant weights give rise to the an error exponent in the lattice point counting problem stated in Theorem  \ref{thm:Best_estimate}.\end{proof}
\begin{rem}
A similar result holds for $n$ even, namely for all dominant weights 
 $$\mu\in\left\{ \lambda_{i}+\lambda_{n+1-i}:i\in\left\{ 1,...,\lfloor\frac{n}{4}\rfloor\right\} \right\} =W_{n}$$
condition 2 of theorem
\ref{thm:Best_estimate} is satisfied with $\frac{n}{2}+1\in I\cap I^{\prime}$. 
\end{rem}

Finally, let us note that $\lambda_{1}+\lambda_{n}=\beta$ is the highest weight of
the adjoint representation, and this completes the proof of  Theorem \ref{thm:adjoin_estimate}.\qed


\subsection{Euclidean norm balls associated with an arbitrary irreducible representation}

The error estimate given by Theorem  \ref{thm:Best_estimate} is the best that our method can provide. As stated in Theorem  \ref{thm:adjoin_estimate}, it arises for the balls associated with the adjoint representation, among others. 
Let us now note that for other irreducible representations it is still possible to improve the error estimate beyond the exponent $\kappa_0$ established in \cite{DRS}. We will prove such an improvement for irreducible representations $\tau$ when the highest weight $\lambda\in \Lambda^+$ 
	belongs to the following set:
	$$
	\Lambda^+_*= \left\{\lambda\in\Lambda^+:\, \exists i\ne 1,n:\, \frac{\lambda(\tilde \beta_i)}{\psi(\tilde \beta_i)}\le \frac{\lambda(\tilde \beta_j)}{\psi(\tilde \beta_j)}\; \hbox{for $j=1,\ldots,n$}\right\}.
	$$
	We note that this set is a union of cones
	$$
	\Lambda^+_*=\bigcup_{i=2}^{n-1}\Lambda^+_i
	$$
	where
		$$
		\Lambda^+_i= \left\{\lambda\in\Lambda^+:\, \frac{\lambda(\tilde \beta_i)}{\psi(\tilde \beta_i)}\le \frac{\lambda(\tilde \beta_j)}{\psi(\tilde \beta_j)}\; \hbox{for $j=1,\ldots,n$}\right\}.
		$$

	\begin{thm}
		Let $G=\hbox{\rm SL}(n+1,\mathbb{R})$ with $n\ge 2$, let $\Gamma$ be any lattice subgroup, and $T\ge T_0$.
		Then for irreducible represenations $\tau$ with highest weight $\lambda\in \Lambda^+_i$, 
		$$
		\frac{|B_T^\tau\cap \Gamma|}{\hbox{\rm vol}(B_T^\tau)}=1+O(\hbox{\rm vol}(B_T^\tau)^{\sigma_i\kappa_0}(\log T)^q)
		$$
		where $\sigma_i=\min(\frac{n}{i},\frac{n}{n+1-i})$.
		Here $q$ and $T_0$ are positive numbers which depends on $n$ but not on the lattice, and can be made explicit. 

	\end{thm}
	
	\begin{proof}
		We recall that
		$$
		\frac{|B_T^\tau\cap \Gamma|}{\hbox{\rm vol}(B_T^\tau)}=1+O\left(\hbox{\rm vol}(B_T^\tau)^{-2n\left(1-\frac{m_1}{m_1'}\right)\kappa_0}(\log T)^q\right)
		$$
		where 
		$$
		m_1=\min_j \frac{\lambda(\tilde \beta_j)}{2\rho(\tilde \beta_j)}\quad\hbox{and}\quad
		m_1'=\min_j \frac{\lambda(\tilde \beta_j)}{\psi(\tilde \beta_j)}.
		$$
		Since $\lambda\in \Lambda^+_i$,
		\begin{align*}
		\frac{m_1}{m_1'}\le \frac{\lambda(\tilde \beta_i)}{2\rho(\tilde \beta_i)}\cdot 
		\left(\frac{\lambda(\tilde \beta_i)}{\psi(\tilde \beta_i)}\right)^{-1}
		=\frac{\psi(\tilde \beta_i)}{2\rho(\tilde \beta_i)}=1-\frac{\gamma(\tilde \beta_i)}{2\rho(\tilde \beta_i)}.
		\end{align*}	
		The last expression is symmetric with respect to $i\mapsto n+1-i$.
		so that we may assume that $i\le (n+1)/2$. Then for $i\le (n+1)/2$,
		$$
		\frac{m_1}{m_1'}\le 1-\frac{i/2}{i(n+1-i)}=1-\frac{1}{2(n+1-i)}.
		$$
		This implies the theorem.
	\end{proof}
	\begin{rem}
	\begin{enumerate}
\item Note that the theorem gives a non-trivial improvement over the bound $\kappa_0$ provided $i\ne 1,n$.
\item The best improvement is achieved when $i=(n+1)/2$ for odd $n$ and 
	when $i=n/2$ or $i=n/2+1$ for even $n$.
	\item 	It follows from Theorem 10 that the cones $\Lambda^+_{(n+1)/2}$ for odd $n$, and 
	$\Lambda^+_{n/2}$ and $\Lambda^+_{n/2+1}$ for even $n$ are not empty.
	\end{enumerate}
\end{rem}


\begin{thebibliography}{10}

\bibitem{CHH}
M.~Cowling, U.~Haagerup, and R.~Howe.
\newblock Almost l2 matrix coefficients.
\newblock {\em J. reine angew. Math}, 387:97--110, 1988.

\bibitem{DRS}
W.~Duke, Z.~Rudnick, and P.~Sarnak.
\newblock Density of integer points on affine homogeneous varieties.
\newblock {\em Duke Math}, 71 no. 1:143--179, 1993.

\bibitem{Folland}
B.~Folland.
\newblock {\em A course in abstract harmonic analysis}.
\newblock CRC Press, Florida, 1995.

\bibitem{GV}
R.~Gangolli and V.~S. Varadarajan.
\newblock {\em Harmonic analysis of spherical functions on real reductive
  groups}, volume 101.
\newblock Springer-Verlag, modern surveys in mathematics, 1988.

\bibitem{GN10}
A.~Gorodnik and A.~Nevo.
\newblock {\em The ergodic theory of lattice subgroups}, volume 172.
\newblock Annals of Mathematics Studies, Princeton University Press, 2010.

\bibitem{GN}
A.~Gorodnik and A.~Nevo.
\newblock Counting lattice points.
\newblock {\em J. Reine Angew. Math}, 663:127--176, 2012.

\bibitem{GW}
A.~Gorodnik and B.~Weiss.
\newblock Distribution of lattice orbits on homogeneous varieties.
\newblock {\em Geom. Funct. Anal}, 17 no. 1:58--115, 2007.

\bibitem{HoweTan}
R.~Howe and E.-C. Tan.
\newblock {\em Nonabelian harmonic analysis. Applications of SL(2,R).}
\newblock Universitext. Springer-Verlag, New York, 1992.

\bibitem{Humphreys}
J.~E. Humphreys.
\newblock {\em Introduction to Lie algebras and representation theory}.
\newblock Springer-Verlag, New York, 1972.

\bibitem{Knapp}
A.~W. Knapp.
\newblock {\em Representation theory of semisimple groups an overview based on
  examples}.
\newblock Princeton University Press, 41 Wiliiam Street, 1986.

\bibitem{Knapp-Green}
A.~W. Knapp.
\newblock {\em Lie groups beyond an introduction}.
\newblock Brikhauser, Boston, 1996.

\bibitem{M}
F.~Maucourant.
\newblock Homogeneous asymptotic limits of haar measures of semisimple lie
  groups and their lattices.
\newblock {\em Duke Mathematical Journal}, 136 (2):357--399, 2007.

\bibitem{Oh}
H.~Oh.
\newblock Tempered subgroups and representations with minimal decay of matrix
  coefficients.
\newblock {\em Bull. Math. Soc. France}, 126:355--380, 1998.

\bibitem{Wolf}
J.~A. Wolf.
\newblock {\em Harmonic analysis on commutative spaces}, volume 142.
\newblock Mathematical surveys and monographs, 1936.

\end{thebibliography}
\end{document}